\documentclass[a4papere,12pt]{article}
\usepackage{amsmath}
\usepackage{amsthm}
\usepackage{amssymb}
\usepackage{amscd}

\newtheorem{definition}{Definition}[section]
\newtheorem{theorem}[definition]{Theorem}
\newtheorem{proposition}[definition]{Proposition}
\newtheorem{lemma}[definition]{Lemma}
\newtheorem{cor}[definition]{Corollary}

\theoremstyle{definition}
\newtheorem{example}[definition]{Example}
\newtheorem{remark}[definition]{Remark}

\newcommand{\Aut}{\mathrm{Aut}}
\newcommand{\K}{\mathcal{K}}
\newcommand{\eS}{\mathcal{S}}
\newcommand{\eL}{\mathcal{L}}
\newcommand{\N}{\mathbb{N}}
\newcommand{\infgal}{\mathrm{Inf}\text{-}\mathrm{gal}\, }

\title{On the definition of the Galois group of linear differential equations}
\author{Katsunori Saito\\
{\normalsize Graduate School of Mathematics, Nagoya University}\\
{\footnotesize e-mail: m07026e@math.nagoya-u.ac.jp}
}
\date{}

\begin{document}
\maketitle
\begin{abstract}
Let us consider a linear differential equation $Y'=AY$ over a differential field $K$, where $A \in \mathrm{M_{n}}(K)$. 
For a differential field extension $L/K$ generated by a fundamental system $F$ of the equation, 
we show that Galois group according to the general Galois theory of Umemura coincides with the Picard-Vessiot Galois group. 
This conclusion generalized the comparision theorem of Umemura and Casale. 
\end{abstract}
\section{Introduction}
There are two ways of understanding the Galois theory of algebraic equations. 
In the first way, when we are given an algebraic equation, according to Galois's original paper, the Galois group is described as a permutation group of solutions of the equation. 

In the second way, Dedekind introduced a nice idea, an algebraic equation over field $k$ determines the normal extension $F/k$ of fields. The Galois group is attached to the normal extension. 
Namely the Galois group is the automorphism group $\Aut(F/k)$ of the field extension. He replaced algebraic equations by field extension. 

The similar situation arises in Picard-Vessiot theory. We consider a linear differential equation
\begin{equation}
Y' = AY \label{ldeq}
\end{equation}
over differential field $K$ with algebraically closed field of constants $C$. 
In the first approach, as Dedekind did, we define the normal extension $L^{PV}/K$ of the system of linear differential equation \eqref{ldeq} that is uniquely determined, called the Picard-Vessiot extension of \eqref{ldeq}. Then the field of constants of $L^{PV}$ is equal to $C$ and the Galois group $\mathrm{Gal}(L^{PV}/K)$ of system \eqref{ldeq} is the differential automorphism group $\Aut^{\partial}(L^{PV}/K)$ that has an algebraic group structure over $C$. 

In the second approach, linear differential equation \eqref{ldeq} generates a neutral tannakian category of differential modules. 
The Galois group of system \eqref{ldeq} is the affine group scheme of the neutral tannakian category. 

In this note, we consider a differential field $L$ over $K$ generated by a system of solutions of linear differential equation \eqref{ldeq}. The field $L$ satisfies the following conditions. (1) There exists a fundamental matrix $F \in \mathrm{GL}_{n}(L)$ so that $F' = AF$. (2) $L$ is generated over $K$ by the entries $f_{ij}$ of $F$. 
We consider Galois group $\infgal(L/K)$ of the differential field extension $L/K $ in general Galois theory of Umemura, 
and compare to Galois group $\mathrm{Gal}(L^{PV}/K)$ of the Picard-Vessiot extension $L^{PV}/K$ for equation \eqref{ldeq}. 
Then we get the following main theorem (Theorem \ref{1210b}). 
\begin{theorem}
We assume that the field $K$ is algebraically closed. 
Let $S$ be a differential domain generated over $K$ by a system of solutions of a linear differential equation $Y' = AY$ and $R= K[Z^{PV}, \, (\det Z^{PV})^{-1}]$ be a Picard-Vessiot ring for the equation $Y'=AY$. 
The field $L=Q(S)$ is the field of fractions of $S$ and the field $L^{PV} = Q(R)$ is the field of fractions of $R$. 
Then we have an isomorphism
\[
\mathrm{Lie}\;(\infgal (L/K)) \simeq \mathrm{Lie}\;(\mathrm{Gal}(L^{PV}/K)) \otimes_{C} L^{\natural}. 
\]
\end{theorem}
Theorem also holds for a $G$-primitive extension. 

In the case of the field $L$ is Picard-Vessiot field, that is $C_{L}=C_{K}$, Umemura \cite{umemura3} prove similar theorem. 
Casale seems to have proved the theorem for $\mathrm{tr.d.} L/K = n^{2}$, where $A$ and $F$ are of $n \times n$ degree matrices.

We are inspired of \cite{umemura3} but we considerably simplified the argument there.

This paper is organized as follows. In \S 2, we give some definitions and results on Picard-Vessiot theory. In \S 3, we give definitions of general Galois theory. In \S 4 we prove the main results. 

The author is grateful to his advisor of thesis Associate Professor Y. Ito who taught him algebraic geometry, indispensable to study algebraic differential equations. He also expresses his thanks to Professor Y. Kimura. Discussions with him on the integrability question of vortices was quite useful to understand the importance of Picard-Vessiot theory. His thanks go also to H. Umemura who guided him to the subject, for valuable discussions. 
\section{Picard-Vessiot theory}\label{1202g}
In this section, we recall the results in Picard-Vessiot theory that is Galois theory of linear differential equations. For more details, we refer to \cite{put-singer}. 
All the ring that we consider except for Lie algebras are commutative and unitary $\mathbb{Q}$-algebra. 

\begin{definition}
A derivation on a ring $R$ is a map $\partial : R \rightarrow R$ satisfying the following properties. 
\begin{enumerate}
\renewcommand{\labelenumi}{(\arabic{enumi})}
\item $\partial(a+b) = \partial(a) + \partial(b)$
\item $\partial(ab) = \partial(a)b + a\partial(b)$
\end{enumerate}
for all $a,\,b \in R$. 

We call a ring $R$ equipped with a derivation $\partial$ on $R$ a differential ring and similarly a field $K$ equipped with a derivation $\partial$ a differential field. 
\end{definition}
We say differential ring $S$ is a differential extension of the differential ring $R$ or a differential ring over $R$ if the ring $S$ is an over ring of $R$ and the derivation $\partial_{S}$ of $S$ restricted on $R$ coincide with the derivation $\partial_{R}$ of $R$. 
We will often denote a differential ring equipped with derivation $\partial$ by $(R, \partial)$ and 
$\partial(a)$ by $a'$. A derivation $\partial$ will be sometimes called a differentiation. 
\begin{example} The following rings are differential rings. \\
1. The polynomial ring $\mathbb{Q}[x]$ over $\mathbb{Q}$ with derivation $f \mapsto f'=df/dx$. \\
2. The ring of formal power series $\mathbb{C}[[x]]$ over $\mathbb{C}$ with derivation $f \mapsto f'=df/dx$. \\
3. The ring of holomorphic functions $\mathcal{O}(U)$ on an open connected subset $U \subset \mathbb{C} $ with derivation $f \mapsto f'= df/dz$. \\
\end{example}
\begin{example} The following fields are differential fields. \\
1. The rational function field $\mathbb{Q}(x)$ over $\mathbb{Q}$ with derivation $f \mapsto f'=df/dx$. \\
2. The field of formal Laurent series $\mathbb{C}[[x]][x^{-1}]$ over $\mathbb{C}$ with derivation $f \mapsto f'=df/dx$. \\
3. The field of meromorphic functions $\mathcal{M}(U)$ on an open connected subset $U \subset \mathbb{C} $ with derivation $f \mapsto f'= df/dz$. \\
In general the field of fractions of a differential domain is a differential field. 
\end{example}
\begin{definition}
Let $(R,\, \partial)$ is a differential ring. An element $c \in R$ is called a constant if $c'=0$ and
a set $C_{R}$ denotes the set of all constants of $R$. 
\end{definition}
By definition, the set of constants $C_{R}$ of $(R,\, \partial)$ forms a ring. Similarly $C_{K}$ is a field for a differential field $(K,\, \partial)$. 
We sometimes say the ring of constants $C_{R}$ or the field of constants $C_{R}$. 
\begin{example}
1. The set of constants of differential ring $(\mathbb{Q}[x],\, d/dx)$ is $\mathbb{Q}$. \\
2. The set of constants of differential field $(\mathbb{C}(z,\, \exp z),\, d/dz)$ is $\mathbb{C}$. \\
3. We consider a differential ring $(\mathbb{C}(x,\, y),\, \partial)$. The derivation $\partial = \partial/\partial x + \partial /\partial y$ acts as follows. $\partial(\mathbb{C}) = 0,\, \partial(x) = 1,\, \partial(y) = 1$. 
Since $\partial (x -y) = \partial(x) - \partial (y) = 1-1=0$, the constants field of $\mathbb{C}(x,\, y)$ is $\mathbb{C}(x-y)$. 
\end{example}
\begin{definition}
A differential ideal $I$ of a differential ring $(R,\, \partial)$ is an ideal of $R$ satisfying $f' \in I$ for all $f \in I$. 
A simple differential ring is a differential ring whose differential ideals are only $(0)$ and $R$. 
\end{definition}
{\it From now on} let $K$ be a differential field with derivation $\partial$ and we assume that{\it the field $C_{K}=C$ of constants of the base field $K$ is algebraically closed. }
We consider a linear differential equation 
\[
Y'= AY,\quad A \in M_{n}(K), 
\]
where $Y = (y_{ij})$ is a $n \times n$ matrix of variables $y_{ij}$ and $Y'=(y'_{ij})$. 
\begin{definition}
A Picard-Vessiot ring $(R,\, \partial)$ over $K$ for the equation $Y'=AY$ with $A \in M_{n}(K)$ is a differential ring $R$ over $K$ satisfying:
\begin{enumerate}
\renewcommand{\labelenumi}{(\arabic{enumi})}
\item $R$ is a simple differential ring.
\item There exists a fundamental matrix $F = (f_{ij})\in \mathrm{GL}_{n}(R)$ for $Y'=AY$ so that $F' = (f'_{ij})= AF$. 
\item $R$ is generated as a ring over $K$ by the entries $f_{ij}$ of $F$ and the inverse of the determinant of $F$, i.e., $R=K[f_{ij},\, (\det F)^{-1}]$. 
\end{enumerate}
\end{definition}
\begin{lemma}[Lemma 1.17 \cite{put-singer}]\label{1207a}
Let $R$ be a simple differential ring over $K$. 
\begin{enumerate}
\renewcommand{\labelenumi}{(\arabic{enumi})}
\item $R$ has no zero divisors. 
\item Suppose that $R$ is finitely generated over $K$, then the field of fractions of $R$ has $C$ as a set of constants. 
\end{enumerate}
\end{lemma}
By Lemma \ref{1207a}, 
Picard-Vessiot ring $R$ is a domain and the field $Q(R)$ of fractions of $R$ has $C$ as the constants field. The following Proposition says existence and uniqueness of a Picard-Vessiot ring. 
\begin{proposition}[Proposition 1.20 \cite{put-singer}]\label{1202e}
Let $Y'=AY$ be a matrix differential equation over $K$. 
\begin{enumerate}
\renewcommand{\labelenumi}{(\arabic{enumi})}
\item There exists a Picard-Vessiot ring $R$ for the equation. 
\item Any two Picard-Vessiot rings for the equation are isomorphic. 
\item The field of constants of the quotient field $Q(R)$ of a Picard-Vessiot ring is $C$. 
\end{enumerate}
\end{proposition}
We also consider a Picard-Vessiot field. 
\begin{definition}
A Picard-Vessiot field for the equation $Y'= AY$ over $K$ is the field of fractions of Picard-Vessiot ring for this equation. 
\end{definition}
The following Proposition characterizes a Picard-Vessiot field. 
\begin{proposition}[Proposition 1.22 \cite{put-singer}]\label{1202i}
Let $Y'=AY$ be a matrix differential equation over $K$ and let $L \supset K$ be an extension of differential fields. 
The field $L$ is a Picard-Vessiot field for this equation if and only if the following conditions are satisfied.
\begin{enumerate}
\renewcommand{\labelenumi}{(\arabic{enumi})}
\item The field of constants of $L$ is $C$. 
\item There exist a fundamental matrix $F \in \mathrm{GL}_{n}(L)$ for the equation, and 
\item the field $L$ is generated over $K$ by the entries of $F$. 
\end{enumerate}
\end{proposition}
We give the following examples of Picard-Vessiot ring and field. 
\begin{example}\label{1209a}
We work over $(K, \, \partial)= (\mathbb{C}(x),\, d/dx)$. We consider a linear differential equation 
\begin{equation}\label{1209d}
\frac{dy}{dx} = y. 
\end{equation}
We take a differential extension ring $R := \mathbb{C}(x)[\exp x, (\exp x)^{-1}]$ over $K$ and the field of fractions $L:=Q(R)= \mathbb{C}(\exp x)$ of $R$. 
Since the field $L$ satisfies the three conditions of Proposition \ref{1202i}, the field $L$ is a Picard-Vessiot field and the ring $R$ is a Picard-Vessiot ring for equation \eqref{1209d}. 
\end{example}
\begin{example}\label{1209b}
We work over $(K, \, \partial)= (\mathbb{C}(x),\, d/dx)$. We consider a linear differential equation
\begin{equation}\label{1209c}
y'' + \frac{1}{x} y' = 0. 
\end{equation}
We denote this equation by the matrix equation
\begin{equation}\label{1209e}
Y' = 
\begin{pmatrix}
0 & 1\\
0 & -\frac{1}{x}
\end{pmatrix}Y. 
\end{equation}
We take two solutions $y_{1} = \log x$ and $y_{2} = 1$ of equation \eqref{1209c} and set
\[
F := \begin{pmatrix}
y_{1} & y_{2} \\
y'_{1} & y'_{2}
\end{pmatrix} = \begin{pmatrix}
\log x & 1\\
\frac{1}{x} & 0
\end{pmatrix}. 
\]
Since $\det F = 1/x \neq 0$, the matrix $F$ is a fundamental matrix of equation \eqref{1209e}. 
So we take a differential ring 
\[
R := K[F,\, (\det F)^{-1}] = \mathbb{C}(x)[\log x ,\, x, \, 1,\, 0,\, x] = \mathbb{C}(x)[\log x]
\]
and the field of fractions \[
L:=Q(R)= \mathbb{C}(x)(\log x)
\] of $R$. 
Then the field $L$ satisfies the three conditions of Proposition \ref{1202i}. So $L$ is a Picard-Vessiot field and the ring $R$ is a Picard-Vessiot ring for equation \eqref{1209c}. 
\end{example}

We define the Galois group for Picard-Vessiot extension. 
\begin{definition}
The differential Galois group of an equation $Y'=AY$ over $K$ is defined as the group $\mathrm{Gal}(L/K)$ 
of differential $K$-automorphism of a Picard-Vessiot field $L$ for the equation. 
\end{definition}
We can think of the differential Galois group $\mathrm{Gal}(L/K)$ as a subgroup of $\mathrm{GL}_{n}(C)$ in the following manner. 
Let $F \in \mathrm{GL}_{n}(L)$ be a fundamental matrix so that $F'= AF$. 
The image of $\sigma(F)$ for $\sigma \in \mathrm{Gal}(L/K)$ is also fundamental matrix for the equation. In fact, 
\[
\sigma(F)' = \sigma(F') = \sigma(AF) = \sigma (A) \sigma(F) = A\sigma(F). 
\]
We consider the matrix $C(\sigma) = F^{-1}\sigma(F) \in \mathrm{GL}_{n}(L)$. The matrix $C(\sigma)$ is also in $\mathrm{GL}_{n}(C)$. 
Indeed
\begin{align*}
C(\sigma)' = (F^{-1}\sigma(F) )'& = (F^{-1})' \sigma(F) + F^{-1}\sigma(F)' \\&= -F^{-1}F'F^{-1} \sigma(F) + F^{-1}A\sigma(F) \\&= -F^{-1}(AF) F^{-1}  \sigma(F) + F^{-1}A\sigma(F) = O. 
\end{align*}
So we get the map $\mathrm{Gal}(L/K) \rightarrow \mathrm{GL}_{n}(C), \quad \sigma \mapsto C(\sigma)$ which is an injective group homomorphism. 
Following Theorem says the Galois group has a linear algebraic group structure over the constants field $C$. 
\begin{theorem}[Theorem 1.27 \cite{put-singer}]
Let $Y'=AY$ be a linear differential equation over $K$, having Picard-Vessiot field $L \supset K$ 
and differential Galois group $G=\mathrm{Gal}(L/K)$. Then 
\begin{enumerate}
\renewcommand{\labelenumi}{(\arabic{enumi})}
\item $G$, considered as a subgroup of $\mathrm{GL}_{n}(C)$, is an algebraic group defined over $C$. 
\item The Lie algebra of $G$ coincides with the Lie algebra of the derivations of $L/K$ that commute with the derivation $\partial_{L}$ on $L$.
\item The field $L^{G}$ of $G$-invariant elements of $L$ is equal to $K$. 
\end{enumerate}
\end{theorem}
\begin{example}The Picard-Vessiot extension $\mathbb{C}(x, \, \exp x)/ \mathbb{C}$ for $y'=y$ in Example \ref{1209a}. 
We take a fundamental solution $y =\exp x$ and an element of the Galois group $\sigma \in \mathrm{Gal}(\mathbb{C}(x, \, \exp x)/ \mathbb{C})$. We have 
\[
y^{-1}\sigma(y) = c \in \mathrm{GL}(\mathbb{C}) \simeq \mathbb{C}^{\times} \simeq \mathbf{G}_{m}(\mathbb{C}). \]
The image $\sigma(y) = yc = c \exp x$ is again a fundamental solution of $y'=y$. 
Then we get $\mathrm{Gal}(\mathbb{C}(x, \, \exp x)/ \mathbb{C}) \simeq \mathbf{G}_{m}$. 
So the Galois group of the Picard-Vessiot extension $\mathbb{C}(x, \, \exp x)/ \mathbb{C}$ is the multiplicative group $\mathbf{G}_{m}(\mathbb{C})$
\end{example}
\begin{example}The Picard-Vessiot extension $\mathbb{C}(x)(\log x)/\mathbb{C}(x)$ for $y'' + (1/x) y' = 0. $ in Example \ref{1209b}. 
We take a fundamental matrix
\[
F = \begin{pmatrix}
\log x & 1\\
\frac{1}{x} & 0
\end{pmatrix}. 
\] and an element of the Galois group $\sigma \in \mathrm{Gal}(\mathbb{C}(x)(\log x)/\mathbb{C}(x))$. Then
\[
F^{-1} \sigma (F) = C(\sigma) = \begin{pmatrix}
a & b\\
c & d
\end{pmatrix} \in \mathrm{GL}(\mathbb{C})
\] 
So we have 
\[
\sigma(F) = \begin{pmatrix}
\log x & 1\\
\frac{1}{x} & 0
\end{pmatrix} \begin{pmatrix}
a & b\\
c & d
\end{pmatrix} .
\]Since entries $1/x,\, 1,\, 0$ of $F$ are in $\mathbb{C}(x)$, these entries are invariant under $\sigma$. Hence we get $a = d = 1$ and $b = 0$. 
\[
\mathrm{Gal}(\mathbb{C}(x)(\log x)/\mathbb{C}(x)) \simeq \left\{
\begin{pmatrix}
1 & 0\\
c & 1
\end{pmatrix}
\right\} \simeq \mathbf{G}_{a}(\mathbb{C})
\]
Then the Galois group of the Picard-Vessiot extension $\mathbb{C}(x)(\log x)/\mathbb{C}(x)$ is the additive group $\mathbf{G}_{a}(\mathbb{C})$.
\end{example}
Following Proposition says there exists a Galois correspondence. 
\begin{proposition}[Proposition 1.34 \cite{put-singer}]\label{1202f}
Let $Y'=AY$ be a differential equation over $K$ with Picard-Vessiot field $L$ with Galois group $G := \mathrm{Gal}(L/K)$. We consider the two sets\\
\[
\mathbb{S}:= \{ \text{ closed subgroups of } G \text{ } \}
\] and
\[
\mathbb{L}:= \{ \text{ differential subfields } M \text{ of } L \text{, containing } K \text{ } \}. 
\] \\
We define the map \[
\alpha \colon \mathbb{S} \to \mathbb{L},\; H \mapsto L^{H}= \{ a \in L \,|\, \sigma(a) = a,\, \text{ for every } \sigma \in H \}, 
\] 
and 
\[
\beta \colon \mathbb{L} \to \mathbb{S},\, M \mapsto \mathrm{Gal}(L/M)= \{ \sigma \in G \,|\, \sigma|_{M} = Id_{M} \}. 
\]
Then
\begin{enumerate}
\renewcommand{\labelenumi}{(\arabic{enumi})}
\item The maps $\alpha$ and $\beta$ are mutually inverse. 
\item If the subgroup $H\in \mathbb{S}$ is a normal subgroup of $G$ then $M=L^{H}$ is invariant under $G$. Conversely, if $M \in \mathbb{L}$ is invariant under $G$ then $H=Gal(L/M)$ is a normal subgroup of $G$. 
\item If $H \in \mathbb{S}$ is normal subgroup of $G$ then the canonical map $G \to \mathrm{Gal}(M/K)$ is surjective and has kernel $H$. 
Moreover, $M$ is a Picard-Vessiot field for some linear differential equation over $K$ with Galois group $G/H$. 
\item Let $G^{o}$ denote the identity component of the algebraic group $G$. Then $L^{G^{o}} \supset K$ is a finite Galois extension with Galois group $G/G^{o}$ and is the algebraic closure of $K$ in $L$. 
\end{enumerate}
\end{proposition}
\section{General Galois theory}
\subsection{Universal Taylor morphism}
For a differential ring $(R,\, \partial)$, we will denote the abstract ring $R$ by $R^{\natural}$. 
Let $(R, \{\partial_{1},\, \partial, \cdots , \, \partial_{d} \})$ be a partial differential ring. So 
$\partial_{i}$ are mutually commutative derivations of $R$ such that we have 
\[
[\partial_{i},\,\partial_{j}]= \partial_{i}\partial_{j} - \partial_{j}\partial_{i} = 0,\quad \mathrm{for}\, 1\leq i,j \leq d
\]
For example, the ring of power series 
\[\left(S[[X_{1},\, X_{2}, \cdots , \, X_{d}]],\, \left\{ \frac{\partial}{\partial X_{1}}, \, \frac{\partial}{\partial X_{2}}, \, \cdots ,\, \frac{\partial}{\partial X_{d}}\right\} \right)\]
is a partial differential ring for a $\mathbb{Q}$-algebra $S$. 

We call a morphism 
\begin{align}
(R, \{\partial_{1},\, \partial, \cdots , \, \partial_{d} \}) &   \notag \\
\longrightarrow &\left(S[[X_{1},\, X_{2}, \cdots , \, X_{d}]],\, \left\{ \frac{\partial}{\partial X_{1}}, \, \frac{\partial}{\partial X_{2}}, \, \cdots ,\, \frac{\partial}{\partial X_{d}}\right\} \right) \label{1124d}
\end{align}
of differential ring by a Taylor morphism. 
Among Taylor morphisms \eqref{1124d}, there exists the universal one $\iota_{R}$. 
\begin{definition}
The universal Taylor morphism $\iota_{R}$ is a differential morphism
\begin{align}
(R, \{\partial_{1},\, \partial, \cdots , \, \partial_{d} \}) &   \notag \\
\longrightarrow &\left(R^{\natural}[[X_{1},\, X_{2}, \cdots , \, X_{d}]],\, \left\{ \frac{\partial}{\partial X_{1}}, \, \frac{\partial}{\partial X_{2}}, \, \cdots ,\, \frac{\partial}{\partial X_{d}}\right\} \right)
\end{align}
by setting 
\[
\iota_{R} (a) = \sum_{n \in \N^{d}} \frac{1}{n!} \partial^{n}(a) X^{n}
\]
for an element $a \in R$, where we use the standard notation for multi-index. 
\end{definition}
Then the universal Taylor morphism has following properties. 
\begin{proposition}[Umemura \cite{umemura1} Proposition (1.4)] 
(i) The universal Taylor morphism is a monomorphism. \\
(ii) The universal Taylor morphism is universal among the Taylor morphisms. 
\end{proposition}
\subsection{Galois hull $\eL/\K$}
Let $(L,\, \partial_{L})/(K, \, \partial_{K})$ be a differential field extension. 
We assume that the abstract field $L^{\natural}$ is finitely generated over the abstract field $K^{\natural}$. 
We have the universal Taylor morphism 
\begin{equation}\label{1124e}
\iota_{L} \colon L \rightarrow L^{\natural}[[X]]. 
\end{equation}
We choose a mutually commutative basis $\{D_{1},\, D_{2},\, \cdots, \, D_{d}\}$ of the $L^{\natural}$-vector space
$\mathrm{Der}(L^{\natural}/K^{\natural})$ of $K^{\natural}$-derivations of the abstract field $L^{\natural}$. 
We introduce partial differential field $L^{\sharp} := (L^{\natural},\, \{D_{1},\, D_{2},\, \cdots, \, D_{d}\})$. 
Similarly the derivations $\{D_{1},\, D_{2},\, \cdots, \, D_{d}\}$ operate on coefficients of the ring $L^{\natural}$. Then we introduce $\{D_{1},\, D_{2},\, \cdots, \, D_{d}\}$-differential structure on the ring $L^{\natural}[[X]]$. 
So the ring $L^{\natural}[[X]]$ has the differential structure defined by the differentiation $d/dX$ and the set $\{D_{1},\, D_{2},\, \cdots, \, D_{d}\}$ of derivations. 
We denote this differential ring by
\[
L^{\sharp}[[X]] =\left( L^{\natural}[[X]],\, \frac{d}{dx},\,D_{1},\, D_{2},\, \cdots, \, D_{d}\right). 
\]
We replace the target space $L^{\natural}[[X]]$ of universal Taylor morphism \eqref{1124e} by $L^{\sharp}[[X]]$ so that we have
\[
\iota_{L} \colon L \rightarrow L^{\sharp}[[X]]. 
\]
In the definition below, we work in the differential ring $L^{\sharp}[[X]]$. We denote $L^{\sharp}$ by partial differential field of constant power series so that
\[
L^{\sharp} := \left\{ \left. \sum_{i=0}^{\infty} a_{i}X^{i} \in L^{\sharp}[[X]] \; \right| a_{i}= 0 \; \mathrm{for}\; i \geq 1  \right\}. 
\]
Therefore $L^{\sharp}$ is a partial differential sub-field of $L^{\sharp}[[X]]$. 
\begin{definition}
The Galois hull $\eL/\K$ is a partial differential algebra extension in the partial differential ring $L^{\sharp}[[X]]$, 
where $\eL$ is the partial differential sub-algebra generated by the image $\iota(L)$ and $L^{\sharp}$ in $L^{\sharp}[[X]]$. 
And $\K$ is the partial differential sub-algebra generated by the image $\iota(K)$ and $L^{\sharp}$ in $L^{\sharp}[[X]]$ so that \[
\K = \iota(K).L^{\sharp}.
\]
\end{definition}
\subsection{The functor $\mathcal{F}_{L/K}$ of infinitesimal deformations}
For the partial differential field $L^{\sharp}$, we have the universal Taylor morphism
\begin{equation}\label{1124f}
\iota_{L^{\sharp}} \colon L^{\sharp} \rightarrow L^{\natural}[[W_{1},\, W_{2}, \cdots, W_{d}]] = L^{\natural}[[W]], 
\end{equation}
where we denoted the variables $X_{i}$ in \eqref{1124f} by variables $W_{i}$. 
The morphism \eqref{1124f} gives a differential ring morphism
\begin{align}
\left(L^{\sharp}[[X]], \; \left\{ \frac{d}{dX},\, D_{1},\, D_{2},\, \cdots, \, D_{d} \right\}\right) \hspace{15em}
\notag \\
\longrightarrow  \left( L^{\natural}[[W_{1},\, W_{2}, \cdots, W_{d}]][[X]],\; \left\{ \frac{d}{dX},\, \frac{\partial}{\partial W_{1}},\, \frac{\partial}{\partial W_{2}},\, \cdots, \, \frac{\partial}{\partial W_{d}} \right\}\right). \label{1124g}
\end{align}
Restricting the differential morphism \eqref{1124g} to the differential sub-algebra $\eL$, we get a differential morphism
\begin{equation}\label{1124h}
\iota \colon \eL \rightarrow L^{\natural}[[W,\, X]]. 
\end{equation}
Similarly, for an $L^{\natural}$-algebra $A$, we have the partial differential morphism
\begin{equation}\label{1124i}
L^{\natural}[[W,\, X]] \rightarrow A[[W,\, X]]
\end{equation}
induced of the morphism $L^{\natural}\rightarrow A$. We get the differential morphism 
\begin{equation}
\iota \colon \eL \rightarrow A[[W,\, X]]
\end{equation}
by composing \eqref{1124h} and \eqref{1124i}. 
\begin{definition}
We define the infinitesimal deformation functor
\[
\mathcal{F}_{L/K} \colon (Alg/L^{\natural}) \rightarrow (Sets)
\]
from the category $(Alg/L^{\natural})$ of $L^{\natural}$-algebra to the category $(Sets)$ of sets as the set of infinitesimal deformations of the morphism \eqref{1124h}. So
\begin{align*}
\begin{array}{r}
\mathcal{F}_{L/K}(A)=\{ f \colon \eL \rightarrow A[[W,\,X]] \;| \text{$f$ is a partial differential morphism} \qquad\\
\text{congruent to the morphism $\iota$ modulo nilpotent elements}\quad\\
\text{ such that $f|_{\K} = \iota$ } \}.
\end{array}
\end{align*}
\end{definition}
\subsection{Group functor $\infgal(L/K)$ of infinitesimal automorphisms}
\begin{definition}
The Galois group in general Galois Theory is the group functor
\[
\infgal(L/K) \colon (Alg/L^{\natural}) \rightarrow (Grp)
\]
associating to an $L^{\natural}$-algebra $A$ the automorphism group
\[
\begin{array}{r}
\infgal(L/K)(A) = \{ f \colon \eL \hat{\otimes}_{L^{\sharp}} A[[W]] \rightarrow \eL \hat{\otimes}_{L^{\sharp}} A[[W]] \; | \text{$f$ is a differential} \quad\\
\text{$\K \hat{\otimes}_{L^{\sharp}} A[[W]]$-automorphism continuous with respect to the $W$-adic }\;\, \\
\text{topology and congruent to the identity modulo nilpotent elements} \}.
\end{array}
\]
\end{definition}
Then the group functor $\infgal(L/K)$ operates on the functor $\mathcal{F}_{L/K}$. The operation $(\infgal(L/K),\, \mathcal{F}_{L/K})$ is a principal homogeneous space. See Theorem (5.11) \cite{umemura2}. 
\section{Galois group of $L/K$ generated by a linear differential equation over $K$ in general Galois theory}
Let $(K, \, \partial_{K})$ be a differential field of characteristic 0 with an algebraically closed field of constants $C$. 
We consider a differential domain $(S,\, \partial_{S})$ that is an over ring of a differential field $(K, \, \partial_{K})$ satisfying the following conditions. 
\begin{enumerate}
\item There exist $n\times n$ matrices $Z = (z_{ij}) \in GL_n(S)$ and $A \in M_{n}(K)$ such that $Z$ is a solution of a linear differential equation $Y' = AY$. 
\item $S$ is generated as a ring by the entries $z_{ij}$ and the inverse $(\det Z )^{-1}$of the determinant $Z$ over $k$, i.e., $S = k[z_{ij}, (\det Z)^{-1}]_{0\leq i,j \leq n}$. 
\end{enumerate}
We will say the differential domain $(S, \, \partial_{S})$ is generated over $K$ by a system of solutions of the linear differential equation $Y' = AY$ or generated by a linear differential equation over $K$. Similarly we say that the field of fractions $L:=Q(S)$ is generated by a linear differential equation over $K$. 

A Picard-Vessiot ring is a special case of above ring. 
We know that the Picard-Vessiot ring $R$ over $K$ is a domain and the constants of $R$ is equal to $C$. 
There is no increase of constants, more over this is also true that the field of fractions $Q(R)$ so that the field of constants equals $C$. 
However, it is not always true that the ring of constants of $S$ equals $C$. 
\begin{example}
We work over a differential field $(\bar{\mathbb{Q}}(e^{t}),\, \partial_{t})$. We consider a differential domain 
\[
\bar{\mathbb{Q}}(e^{t})[\pi e^{t},\, (\pi e^{t})^{-1}] = \bar{\mathbb{Q}}(e^{t})[\pi ,\, \pi^{-1}]
\]
with derivation $\partial_{t}$. 
Hence the differential ring $\bar{\mathbb{Q}}(e^{t})[\pi ,\, \pi^{-1}]$ is generated by a linear differential equation $ Y' = Y$ over $\bar{\mathbb{Q}}(e^{t})$. Since 
\[
C_{\bar{\mathbb{Q}}(e^{t})[\pi ,\, \pi^{-1}]} = \bar{\mathbb{Q}}[\pi ,\, \pi^{-1}] \supsetneq \bar{\mathbb{Q}} = C_{\bar{\mathbb{Q}}(e^{t})}, 
\]
the over-ring $\bar{\mathbb{Q}}(e^{t})[\pi ,\, \pi^{-1}]$ has more constant than the base field $\bar{\mathbb{Q}}(e^{t})$.
\end{example}
So we can not treat the differential domain $S$ generated by a linear differential equation over $K$ in Picard-Vessiot theory. 
We will compare the Galois group of generated over $k$ by a system of solutions of the linear differential equation 
\begin{equation}\label{ldeq2}
Y'= AY, \qquad A \in M_{n}(K), 
\end{equation}
and the Galois group of Picard-Vessiot extension for the equation \eqref{ldeq2} in general Galois theory. 
If the ring $S$ is a Picard-Vessiot ring, by Proposition \ref{1202f} (4) in Section \ref{1202g} we get an isomorphism
\[
\mathrm{Lie} (\mathrm{Gal}(Q(S)/K)) \simeq \mathrm{Lie} (\mathrm{Gal}(Q(S)^{G^{o}}/K)). 
\]
As we interested in the Lie algebra of the Galois group, replacing the base field $K$ by its algebraic closure in $Q(S)$. 

Let $S$ be generated by linear differential equation \eqref{ldeq2}. And the field $L = Q(S)$ is the field of fractions of $S$. Then there exists $n\times n$ matrix $Z = (z_{ij}) \in \mathrm{GL}_{n}(L)$ such that
\begin{equation} \label{1124j}
Z' = AZ
\end{equation}
and $L = k(z_{ij})$. 
In the following, we work in the differential ring 
$L^{\sharp}[[X]]$. 
Since the universal Taylor morphism $\iota$ is a differential morphism, the image of the matrix $Z$ by the $\iota$ satisfies
\begin{equation}\label{1124a}
\frac{d}{dX}(\iota (Z)) = \iota(A)\iota(Z)
\end{equation}
by \eqref{1124j}. 
\begin{lemma}\label{1123a}
If we set $B := \iota(Z) (Z^{\sharp})^{-1} \in \mathrm{GL}_{n}(L^{\sharp}[[X]])$, 
where $Z^{\sharp} = (z_{ij}^{\sharp})$, then the matrix $B$ is in $\mathrm{GL}_{n}(K^{\sharp}[[X]])$.
\end{lemma}
\begin{proof}
We write 
\[
\iota(A) = \sum \frac{1}{k!}A_{k}X^{k},\, A_{k} \in \mathrm{M}_{n}(K^{\sharp})
\]
and 
\[
 B = \sum \frac{1}{k!}B_{k}X^{k},\, B_{k} \in \mathrm{M}_{n}(L^{\sharp}).
\] 
It is sufficient to show that $B_{k}$ is in $\mathrm{M}_{n}(K^{\sharp})$ for $k \in \N$. We show this by induction on $k$. 
For $k=0$, indeed $B_{0} = I_{n} \in \mathrm{M_{n}}(K^{\sharp})$ by definition of $B$. Assume $B_{l} \in \mathrm{M_{n}}(K^{\sharp})$ for $l < k$. Since $Z^{\sharp}$ is constant matrix with respect to $d/dX$, 
it is follows from \eqref{1124a} 
\begin{equation}\label{1124b}
\frac{d}{dx} B = \iota (A)B. 
\end{equation}
We rewrite \eqref{1124b}
\begin{equation}\label{1124c}
\dfrac{d}{dx}\left( \sum \frac{1}{k!}B_{k}X^{k}\right) = \left( \sum \frac{1}{k!}A_{k}X^{k}\right)\left( \sum \frac{1}{k!}B_{k}X^{k}\right).
\end{equation}
Comparing coefficients of $X^{k-1}$ of \eqref{1124c}, we get 
\[
B_{k} = \sum_{l+m=k-1}\frac{(k-1)!}{l!m!}A_{l}B_{m} \in \mathrm{M_{n}}(K^{\sharp}). 
\]
\end{proof}
In the construction of Galois hull $\eL$ in general differential Galois theory, we consider a differential field extension $L/K$. 
However, we replace the differential field $L$ by the differential ring $S = k[Z,\, (\det Z)^{-1}]$. 
We consider the restriction of the universal Taylor morphism $\iota$ to the differential sub-algebra $S = K[Z,\, (\det Z)^{-1}]$ of $L$. 
And we replace the Galois hull $\eL$ by the sub-ring $\eS:= \iota(S).L^{\sharp}$ of $L^{\sharp}[[X]]$. 

\begin{lemma}\label{1125a} In the differential ring $L^{\sharp}[[X]]$, 
the differential sub-ring 
\[
\iota(K)[B, (\det B)^{-1}].L^{\sharp}
\]
coincides with the differential sub-ring
\[\eS = \iota(K[Z,\, (\det Z)^{-1}]).L^{\sharp}. \]
\end{lemma}
\begin{proof}
From $B = \iota(Z) (Z^{\sharp})^{-1}$, 
\[
\begin{array}{r}
\eS = \iota(K[Z,\, (\det Z)^{-1}]).L^{\sharp} = \iota(K)[\iota(Z),\, (\det \iota(Z))^{-1}]).L^{\sharp} \hspace{5em}\\
= \iota(K)[BZ^{\sharp},\, (\det BZ^{\sharp})^{-1}]).L^{\sharp} = \iota(K)[B,\, (\det B)^{^1}].L^{\sharp}
\end{array}
\]
\end{proof}
\begin{lemma}\label{1201d}
The sub-ring $\eS$ of $L^{\sharp}[[X]]$ is a differential sub-ring. 
\end{lemma}
\begin{proof}
We show that the ring $\eS$ is closed under the derivations $d/dX$ and $D_{i}$ for $1 \leq i \leq d$. 
Since both $\iota(K[Z,\, (\det Z)^{-1}])$ and $L^{\sharp}$ are closed under the differentiation $d/dX$, the ring $\eS$ is closed under the differentiation. 
To show that the ring $\eS$ closed under the derivations $D_{i}$, by Lemma \ref{1125a}, we will show the ring $\iota(K)[B, (\det B)^{-1}].L^{\sharp}$ is closed under the derivations. 
The sub-ring $L^{\sharp}$ is closed obviously. 
By Lemma \ref{1123a}, the ring $\iota(K)[B, (\det B)^{-1}]$ is in the ring $K^{\sharp}[[X]]$ so that derivations $D_{i}$ act trivially on $\iota(K)[B, (\det B)^{-1}]$. 
Then $\iota(K)[B, (\det B)^{-1}].L^{\sharp}$ closed under the derivations $D_{i}$. 
\end{proof}
From Lemma \ref{1123a} we get $\iota(K)[B,\, (\det B)^{-1}] \subset K^{\sharp}[[X]]$. 
So we have $C \subset C_{\iota(K)[B,\, (\det B)^{-1}] } \subset K^{\sharp}$. 
\begin{example}
We consider a differential ring $\mathbb{C}(x)[\exp x,\, (\exp x)^{-1}]$. 
The ring $\mathbb{C}(x)[\exp,\,(\exp x)^{-1}]$ is a Picard-Vessiot ring over $\mathbb{C}(x)$ for a linear differential equation $Y'=Y$ with the fundamental matrix $Z = \exp x$. 
The image of fundamental matrix $\iota(Z)$ is
\[
\iota(Z) = \exp x + (\exp x )X + \frac{1}{2!}(\exp x )X^{2} + \cdots = \exp(x+X)
\]
Then the matrix $B$ is 
\[
B = \iota(Z) (Z^{\sharp})^{-1} = (\exp (x+X))(\exp x)^{-1} = \exp(X)
\]
So \[
\iota(K)^{\sharp}[B,\, (\det B)^{-1}] = \iota(\mathbb{C}(x))[\exp X,\, (\exp X)^{-1}] 
\]
In this case $C_{\iota(\mathbb{C}(x))[\exp X,\, (\exp X)^{-1}]} = \mathbb{C}$. 
\end{example}
\begin{example}
We consider a differential ring $\mathbb{C}(x)[\log x]$. 
The ring $\mathbb{C}(x)[\log x]$ is a Picard-Vessiot ring over $\mathbb{C}(x)$ for a linear differential equation 
\[
Y' = 
\begin{pmatrix}
0 & 1\\
0 & -\frac{1}{x}
\end{pmatrix}Y
\] with the fundamental matrix \[
Z =
\begin{pmatrix}
\log x & 1\\
\frac{1}{x} & 0
\end{pmatrix}. 
\]
The image of fundamental matrix $\iota(Z)$ is
\[
\iota(Z) = 
\begin{pmatrix}
\log x + \frac{1}{x}X -\frac{1}{2x}X^{2} + \cdots  & 1\\
\frac{1}{x} - \frac{1}{x^{2}}X + \frac{1}{x^{3}}X^{2} + \cdots & 0
\end{pmatrix} = \begin{pmatrix}
\log (x + X) & 1\\
\frac{1}{x+X} & 0
\end{pmatrix}. 
\] 
So the matrix $B$ is 
\[
B = \iota(Z)(Z^{\sharp})^{-1} = \begin{pmatrix}
\log (x + X) & 1\\
\frac{1}{x+X} & 0
\end{pmatrix} \begin{pmatrix}
0& x\\
1 & -x\log x
\end{pmatrix} = \begin{pmatrix}
1 & x(\log (1+\frac{X}{x})\\
0 & \frac{x}{x+X}
\end{pmatrix}. 
\] 
Then we get,
\[
x = \left(\frac{x}{x+X} \right) (x+X) \in \iota(\mathbb{C}(x))[B,\, (\det B)^{-1}]
\]
Since $x$ is a constant with respect to derivations $d/dX$ and $D_{i}$, the ring of constants
 $\mathbb{C}(x)^{\sharp}[B, (\det B)^{-1}]$ is larger than the field $\mathbb{C}$. 
\end{example}
Then we consider the sub-ring $\iota(K)[B,\, (\det B)^{-1}].K^{\sharp}$ of $\eS$. The sub-ring is also a differential sub-ring. 
The following lemma is a famous result called linear disjointness theorem. 
\begin{lemma} [Kolchin]\label{1126a}
Let $(R,\, \partial)$ be a differential ring and $M$ be a differential sub-field of $R$. 
Then the field $M$ and the ring of constants $C_{R}$ of $R$ are linearly disjoint over the field of constants $C_{M}$ of $M$. 
\end{lemma}
\begin{proof}
See \cite{umemura1} Lemma (1.1)
\end{proof}
The lemma above as well as the lemma below are quite useful. 
\begin{lemma} \label{1201a}
Let $M$ be a field and $(M[[X]],\, d/dX)$ be the differential ring of power series with coefficients in $M$. Let $R$ be a differential sub-ring of $M[[X]]$ containing the field $M$. 
Then the ring $R$ is a domain and the field of fractions $Q(R)$ has a differential field structure and we have
\[
C_{Q(R)} = M. 
\]
\end{lemma}
\begin{proof}
Since $M[[X]]$ is a domain, it is clear that the sub-ring $R$ is a domain. The field of fractions $Q(R)$ is a sub-field of $Q(M[[X]])=M[[X]][X^{-1}]$ and contains $M$. So,
\[
M \subset C_{Q(R)} \subset C_{M[[X]][X^{-1}]} = M. 
\]
\end{proof}
\begin{remark}
Lemma \ref{1126a} and Lemma \ref{1201a} are also true if rings are partial differential ring. 
\end{remark}
Applying Lemma \ref{1201a} to $\iota(K)[B,\, (\det B)^{-1}].K^{\sharp} \subset K^{\sharp}[[X]]$ and $\eS \subset L^{\sharp}[[X]]$, we have following corollaries. 
\begin{cor}\label{1201b}
The field of constants $C_{Q(\iota(K)[B,\, (\det B)^{-1}].K^{\sharp})}$ 
of the field of fractions $Q(\iota(K)[B,\, (\det B)^{-1}].K^{\sharp})$ of $\iota(K)[B,\, (\det B)^{-1}].K^{\sharp}$ is $K^{\sharp}$
\end{cor}
\begin{cor}\label{1201c}
The field of constants of the differential field $(Q(\eS),\, d/dX)$ is $L^{\sharp}$. 
\end{cor}
\begin{lemma}\label{1202a}
The sub-ring $\iota(K)[B,\, (\det B)^{-1}].K^{\sharp}$ of $(\eS,\, d/dX)$ and the sub-field
 $L^{\sharp}$ are linearly disjoint over $K^{\sharp}$. 
So we have a $d/dX$-differential isomorphism 
\[
\iota(K)[B,\, (\det B)^{-1}].K^{\sharp} \otimes_{K^{\sharp}} L^{\sharp} \simeq \eS
\]
\end{lemma}
\begin{proof}
We work in the differential field $Q(\eS)$. To apply Lemma \ref{1126a} to the differential sub-field \[Q(\iota(K)[B,\, (\det B)^{-1}].K^{\sharp})\]
of $Q(\eS)$, the field $Q(\iota(K)[B,\, (\det B)^{-1}].K^{\sharp})$ and $C_{Q(\eS)}$ are linearly disjoint over $C_{Q(\iota(K)[B,\, (\det B)^{-1}].K^{\sharp})}$. 
So the differential sub-ring $\iota(K)[B,\, (\det B)^{-1}].K^{\sharp}$ and $C_{Q(\eS)}$ also linearly disjoint over $C_{Q(\iota(K)[B,\, (\det B)^{-1}].K^{\sharp})}$. 
Now Lemma follows from Corollary \ref{1201b} and Corollary \ref{1201c}. 
\end{proof}
From now on, we work in the partial differential ring $L^{\natural}[[W,\, X]]$ and 
identify a sub-ring $R$ of $L^{\sharp}[[X]]$ with its image of the universal Taylor morphism $\iota_{L^{\sharp}}$. 
\begin{proposition}\label{1203d}
In the partial differential ring 
\[
\left( L^{\natural}[[W_{1},\, W_{2}, \cdots, W_{d}]][[X]],\; \left\{ \frac{d}{dX},\, \frac{\partial}{\partial W_{1}},\, \frac{\partial}{\partial W_{2}},\, \cdots, \, \frac{\partial}{\partial W_{d}} \right\}\right),
\]
The sub-ring $\eS.L^{\natural}$ is a partial differential sub-ring. 
So we have a partial differential isomorphism 
\[
\eS.L^{\natural} \simeq  (\iota(K)[B,\, (\det B)^{-1}].K^{\sharp} \otimes_{K^{\sharp}} L^{\sharp}) \otimes_{K^{\natural}} L^{\natural}. 
\]
\end{proposition}
\begin{proof}
Since the universal Taylor morphism $\iota_{L^{\sharp}}$ is differential morphism, 
the sub-ring $\eS$ is closed under the differentiations $d/dX$ and $\partial/\partial W_{i}$ by Lemma \ref{1201d}. 
And the constants power series $L^{\natural}$ is clearly closed under the differentiations. 
So the sub-ring $\eS.L^{\natural}$ is a partial differential sub-ring. 
In the same way as the proof of Lemma \ref{1202a}, we have $\{ \partial/\partial W_{i}\}$-differential isomorphism
\begin{equation}\label{1202b}
\eS.L^{\natural} \simeq \eS \otimes_{K^{\natural}} L^{\natural}. 
\end{equation}
Isomorphism \eqref{1202b} is also $\{ d/dX,\,\partial/\partial W_{i}\}$-differential isomorphism because 
the sub-ring $\eS$ and the sub-field $L^{\natural}$ are closed under differentiations 
$d/dX$ and $\partial/\partial W_{i}$. 
By Lemma \ref{1202a}, we have $\{ d/dX \}$-differential isomorphism
\begin{equation}\label{1202c}
\eS \simeq \iota(K)[B,\, (\det B)^{-1}].K^{\sharp} \otimes_{K^{\sharp}} L^{\sharp}. 
\end{equation}
Since the sub-ring $(\iota(K)[B,\, (\det B)^{-1}].K^{\sharp}$ and $L^{\sharp}$ also closed under differentiations $d/dX$ and $\partial/\partial W_{i}$, 
isomorphism \eqref{1202c} is $\{ d/dX,\,\partial/\partial W_{i}\}$-differential isomorphism. 
So we get $\{ d/dX,\,\partial/\partial W_{i}\}$-differential isomorphism
\begin{equation}\label{1202d}
\eS\otimes_{K^{\sharp}} L^{\natural} \simeq (\iota(K)[B,\, (\det B)^{-1}].K^{\sharp} \otimes_{K^{\sharp}} L^{\sharp}) \otimes_{K^{\natural}} L^{\natural}. 
\end{equation}
Then the Proposition follows from \eqref{1202b} and \eqref{1202d}. 
\end{proof}
\begin{cor}
\[
\K.L^{\natural} \simeq \iota(K).K^{\sharp} \otimes_{K^{\sharp}} L^{\sharp} \otimes_{K^{\natural}} L^{\natural}. 
\]
\end{cor}
\begin{cor}\label{1217a}
For an $L^{\natural}$-algebra $A$, 
\[
\eS .A \simeq \iota(K).K^{\sharp} \otimes_{K^{\sharp}} L^{\sharp} \otimes_{K^{\natural}} A
\]
\end{cor}
\begin{proof}
The proof in Proposition \ref{1203d} works also in these cases. 
\end{proof}
We take a subset $(\iota(K).K^{\sharp})^{\ast}$ of the ring $\iota(K)[B,\, (\det B)^{-1}].K^{\sharp}$. 
The set $(\iota(K).K^{\sharp})^{\ast}$ is a multiplicative set. 
The localization of $\iota(K)[B,\, (\det B)^{-1}].K^{\sharp}$ by $(\iota(K).K^{\sharp})^{\ast}$ is equal to $Q(\iota(K).K^{\sharp})[B,\, (\det B)^{-1}]$. 
\begin{lemma}\label{1202h}
The field of constants $C_{Q(\iota(K).K^{\sharp})}$ of $Q(\iota(K).K^{\sharp})$ and the field of constants $C_{Q(\iota(K).K^{\sharp})[B,\, (\det B)^{-1}]}$ of $Q(\iota(K).K^{\sharp})[B,\, (\det B)^{-1}]$ are equal to $K^{\sharp}$. 
\end{lemma}
\begin{proof}
We can apply Lemma \ref{1201a} to $\iota(K).K^{\sharp} \subset K^{\sharp}[[X]]$ then $C_{Q(\iota(K).K^{\sharp})}=K^{\sharp}$. 
\[
K^{\sharp} \subset Q(\iota(K).K^{\sharp})[B,\, (\det B)^{-1}] \subset K^{\sharp}[[X]][X^{-1}]. 
\]
We have $C_{Q(\iota(K).K^{\sharp})[B,\, (\det B)^{-1}]}=K^{\sharp}$. 
\end{proof}
\begin{lemma}\label{1203b}
$Q(\iota(K).K^{\sharp})[B,\, (\det B)^{-1}]$ is a Picard-Vessiot ring over the field $Q(\iota(K).K^{\sharp})$ for the equation $Y'=\iota(A)Y$ if the field $K$ is algebraically closed. 
\end{lemma}
\begin{proof}
We denote $Q(\iota(K).K^{\sharp})[B,\, (\det B)^{-1}]$ by $T$. By the proof of Lemma \ref{1202h}, 
$T$ in $K^{\sharp}[[X]][X^{-1}]$, them the field of fractions $Q(T)$ of the ring $T$ has also the field of constants $K^{\sharp}$. 
And $Q(T)$ has the fundamental matrix $B$ of the equation $Y'=\iota(A)Y$ and $Q(T)$ is generated over $Q(\iota(K).K^{\sharp})$ by the entries of $B$.
Then by Proposition \ref{1202i} of Section \ref{1202g} the field $Q(T)$ is a Picard-Vessiot field and then $T$ is a Picard-Vessiot ring. 
\end{proof}
\begin{lemma}\label{1203c}
We assume that the field $K^{\sharp}$ is algebraically closed. 
Let $R$ be a Picard-Vessiot ring for the equation $Y'=AY$ over the field $Q(K\otimes_{C}K^{\sharp})$. Then 
$R$ and $Q(\iota(K).K^{\sharp})[B,\, (\det B)^{-1}]$ are differentially isomorphic. 
\end{lemma}
\begin{proof}
By the differential isomorphism
\[
Q(\iota(K).K^{\sharp}) \simeq Q(K\otimes_{C}K^{\sharp}), 
\]
the ring $Q(\iota(K).K^{\sharp})[B,\, (\det B)^{-1}]$ is a differential over ring of $Q(K\otimes_{C}K^{\sharp})$
Then by Lemma \ref{1203b} we consider $Q(\iota(K).K^{\sharp})[B,\, (\det B)^{-1}]$ is a Picard-Vessiot ring over the field $Q(K\otimes_{C}K^{\sharp})$ for the equation $Y'=AY$. So this lemma follows from Proposition \ref{1202e} (2) of Section \ref{1202g}. 
\end{proof}
The following theorem is the main result. 
\begin{theorem}\label{1210b}
We assume that the field $K$ is algebraically closed. 
Let $S$ be a differential domain generated over $K$ by a system of solutions of a linear differential equation $Y' = AY$ and $R= K[Z^{PV}, \, (\det Z^{PV})^{-1}]$ be a Picard-Vessiot ring for the equation $Y'=AY$. 
The field $L=Q(S)$ is the field of fractions of $S$ and the field $L^{PV} = Q(R)$ is the field of fractions of $R$. 
Then we have an isomorphism
\[
\mathrm{Lie}\;(\infgal (L/K)) \simeq \mathrm{Lie}\;(\mathrm{Gal}(L^{PV}/K)) \otimes_{C} L^{\natural}. 
\]
\end{theorem}
\begin{proof}
We have to show
\begin{align*}
\text{Inf-aut}(\eS \hat{\otimes}_{L^{\sharp}}A[[W]]/\K \hat{\otimes}_{L^{\sharp}}A[[W]]) \hspace{15em}\\
\simeq \text{Inf-aut}(K[Z^{PV},\, (\det Z^{PV})]\otimes_{C}A/ K \otimes_{C} A)
\end{align*}
for $A=L^{\natural}[\varepsilon]$ with $\varepsilon^{2}=0$, where Inf-aut$(\mathfrak{A/B})$ denotes the set of $\mathfrak{B}$-automorphisms of $\mathfrak{A}$ congruent to the identity map of $\mathfrak{A}$ modulo nilpotent element. 
The following argument works for every $L^{\natural}$-algebra $A$. 
We have the following isomorphisms
\begin{align}
 K[Z^{PV}, \, (\det Z^{PV})^{-1}] \otimes_{C} A &\simeq (K[Z^{PV}, \, (\det Z^{PV})^{-1}] \otimes_{C} K^{\natural}) \otimes_{K^{\natural}} A \notag\\
&\simeq (K[Z^{PV}, \, (\det Z^{PV})^{-1}] . K^{\natural}) \otimes_{K^{\natural}} A, \label{1217b}
\end{align}
over $K\otimes_{C}A$. 
Given an infinitesimal automorphism $f$ of $L^{PV}\otimes_{C} A$ over $K \otimes_{C} A$, 
by isomorphisms \eqref{1217b} defines an infinitesimal automorphism $\tilde{f}$ of 
\[(K[Z^{PV}, \, (\det Z^{PV})^{-1}] . K^{\natural})\otimes_{K^{\natural}} A\]
over $K\otimes_{C}A$. 
The infinitesimal automorphism $\tilde{f}$ extended to the localization of $(K[Z^{PV}, \, (\det Z^{PV})^{-1}] . K^{\natural}) $ by the multiplicative system $(K\otimes_{C}K^{\sharp})^{\ast}$. 
By Lemma \ref{1203c}, we get an infinitesimal automorphism $\bar{f}$ of 
\[
Q(\iota(K).K^{\sharp})[B,\, (\det B)^{-1}]\otimes_{K^{\natural}}A
\]
over $Q(\iota(K).K^{\sharp})\otimes_{C}A$. 
Since $B$ is a system of solutions of linear differential equation, 
the infinitesimal automorphism $\bar{f}$ induces 
an infinitesimal automorphism $\bar{f}'$ of \[
\iota(K)[B,\, (\det B)^{-1}].K^{\sharp}\otimes_{K^{\natural}}A\]
over $\iota(K).K^{\natural}\otimes_{K^{\sharp}}A$. 
Since the $W_{i}$'s are variable, $\bar{f}'$ defines an infinitesimal automorphism of 
\[\iota(K)[B,\, (\det B)^{-1}].K^{\natural}\otimes_{K^{\sharp}}A[[W]]\]
over $\iota(K).K^{\natural} \otimes_{K^{\sharp}}A[[W]]$. 
Therefore an infinitesimal automorphism of $\eS \otimes_{L^{\sharp}} A[[W]]$ over $\K \otimes_{L^{\sharp}} A[[X]]$ by Lemma \ref{1202a}. 
So consequently an infinitesimal automorphism of $\eS \hat{\otimes}_{L^{\sharp}}A[[W]]$ over $\K \hat{\otimes}_{L^{\sharp}}A[[W]]$. 

To prove the converse, 
we notice that we have
\begin{align*}
\text{Inf-aut}(\eS \hat{\otimes}_{L^{\sharp}}A[[W]]/\K \hat{\otimes}_{L^{\sharp}}A[[W]])\hspace{15em} \\
\simeq \text{Inf-aut}(\iota(K)[B,\, (\det B)^{-1}]K^{\sharp}\otimes_{K^{\sharp}}A/ \iota(K)\otimes_{K^{\sharp}} A)
\end{align*}
In fact, given an $g \in \text{Inf-aut}(\eS \hat{\otimes}_{L^{\sharp}}A[[W]]/\K \hat{\otimes}_{L^{\sharp}}A[[W]])$ so that
\[
g \colon \eS \hat{\otimes}_{L^{\sharp}}A[[W]] \rightarrow \eS \hat{\otimes}_{L^{\sharp}}A[[W]]. 
\]
By corollary \ref{1217a} the restriction $\tilde{g}:= g|_{\iota(K)[B,\, (\det B)^{-1}].K^{\sharp}.A}$ 
to the subalgebra \[
\iota(K)[B,\, (\det B)^{-1}].K^{\sharp}.A \simeq \iota(K)[B,\, (\det B)^{-1}].K^{\sharp}\otimes_{K^{\sharp}}A
\] 
of $\eS \otimes_{L^{\sharp}}A[[W]]$ maps $
\iota(K)[B,\, (\det B)^{-1}].K^{\sharp}\otimes_{K^{\sharp}}A
$ 
to \[
\iota(K)[B,\, (\det B)^{-1}].K^{\sharp}.A \simeq \iota(K)[B,\, (\det B)^{-1}].K^{\sharp}\otimes_{K^{\sharp}}A. \]
Therefore we have a commutative diagram
\[
\begin{CD}
\iota(K)[B,\, (\det B)^{-1}].K^{\sharp}\otimes_{K^{\sharp}}A @ > \tilde{g} >> \iota(K)[B,\, (\det B)^{-1}].K^{\sharp}\otimes_{K^{\sharp}}A\\
@VVV @VVV\\
\eS \hat{\otimes}_{L^{\sharp}}A[[W]] @ >> g > \eS \hat{\otimes}_{L^{\sharp}}A[[W]]. 
\end{CD}
\]
Now by isomorphisms \eqref{1217b}, 
\begin{align*}
\text{Inf-aut}(\eS \hat{\otimes}_{L^{\sharp}}A[[W]]/\K \hat{\otimes}_{L^{\sharp}}A[[W]]) \hspace{15em}\\ \simeq \text{Inf-aut}(K[Z^{PV},\, (\det Z^{PV})]\otimes_{C}A/ K \otimes_{C} A)
\end{align*}
\end{proof}
When the base field $K$ is not algebraically closed, the above argument allows us to prove the following result. 
\begin{theorem}\label{1210a}
Let $S$ be a differential domain generated over $K$ by a system of solutions of the linear differential equation $Y' = AY$ and $R$ be a Picard-Vessiot ring for the equation $Y'=AY$. 
The field $L=Q(S)$ is the field of fractions of $S$ and the field $L^{PV} = Q(R)$ is the field of fractions of $R$. 
There exists a finite field extension $\tilde{L}$ of $L^{\natural}$, we have an isomorphism
\[
\mathrm{Lie}\;(\infgal (L/K)) \otimes_{L^{\natural}}\tilde{L}\simeq \mathrm{Lie}\;(\mathrm{Gal}(L^{PV}/K)) \otimes_{C} \tilde{L}. 
\]
\end{theorem}
\begin{proof}
We replace the base field $K$ by its algebraic closure $\bar{K}$ and $L$ by $L\otimes_{K}\bar{K}$. Then we can apply the argument of proof of Theorem \ref{1210b}. 
\end{proof}

\end{document}